\documentclass[11pt]{amsart}
\usepackage{amsfonts}
\usepackage{amsmath}
\usepackage{amssymb,latexsym}
\usepackage[mathcal]{eucal}
\usepackage{amscd}

\usepackage[pdftex]{graphicx}

\usepackage[pdftex,bookmarks,colorlinks,breaklinks]{hyperref}

\input xy
\xyoption{all}

\newcommand{\Z}{\mathbb{Z}}

\newcommand{\R}{\mathbb{R}}
\newcommand{\C}{\mathbb{C}}
\newcommand{\N}{\mathbb{N}}

\newcommand{\al}{\alpha}

\newcommand{\ga}{\gamma}
\newcommand{\del}{\delta}

\newcommand{\ep}{\epsilon}
\newcommand{\sig}{\sigma}

\newcommand{\om}{\omega}
\newcommand{\Om}{\Omega}

\newcommand{\ol}{\overline}

\newcommand{\br}{\vspace{3 mm}}
\newcommand{\imp}{\Rightarrow}

\newcommand{\id}{{\rm{id}}}

\newcommand{\diam}{{\rm{diam\,}}}

\newcommand{\card}{{\rm{card\,}}}

\newcommand{\Homeo}{\rm{Homeo}}

\swapnumbers
\theoremstyle{plain}
\newtheorem{thm}{Theorem}[section]

\newtheorem{lem}[thm]{Lemma}
\newtheorem{prop}[thm]{Proposition}

\theoremstyle{definition}

\newtheorem{rmk}[thm]{Remark}

\def\g{\gamma}

\newcounter{quotecount}

\newenvironment{psmallmatrix}
  {\left(\begin{smallmatrix}}
  {\end{smallmatrix}\right)}

\begin{document}

\title[A metric minimal PI cascade with $2^{\frak{c}}$ minimal ideals]
{A metric minimal PI cascade with $2^{\frak{c}}$ minimal ideals
}


\author[Eli Glasner and Yair Glasner]{Eli Glasner and Yair Glasner}
\address{Department of Mathematics,
Tel Aviv University, Tel Aviv, Israel}
\email{glasner@math.tau.ac.il}

\address{Department of Mathematics,
Ben-Gurion University of the Negev, BeÕer Sheva, Israel}
\email{yairgl@math.bgu.ac.il}

%


%

\begin{date}
{28 September, 2017}
\end{date}
%



\begin{abstract}
We first improve an old result of McMahon and show that 
a metric minimal flow whose enveloping semigroup contains less than 
$2^{\frak{c}}$ (where ${\frak{c}} ={2^{\aleph_0}}$) minimal left ideals is PI.
Then we show the existence of various minimal PI flows with many minimal left ideals, as follows.
For the acting group $G=SL_2(\R)^\N$,
we construct a metric minimal PI $G$-flow with $\frak{c}$ minimal left ideals.
We then use this example and results established in \cite{GW-79}
to construct a metric minimal PI cascade $(X,T)$ with
$\frak{c}$ minimal left ideals.
We go on and construct an example of a minimal PI-flow $(Y, \mathcal{G})$ 
on a compact manifold $Y$ and a suitable path-wise connected group $\mathcal{G}$
of homeomorphism of $Y$, such that the 
flow $(Y, \mathcal{G})$ is PI and has $2^{\frak{c}}$ minimal left ideals. 
Finally, we use this latter example and a theorem of Dirb\'{a}k to construct
a cascade $(X, T)$ which is PI (of order 3) and has 
$2^\frak{c}$ minimal left ideals.
Thus this final result shows that, even for cascades, the converse 
of the 
implication ``less than $2^\frak{c}$ minimal left ideals implies PI", fails.

\end{abstract}
 


\thanks{The research of E. G. was supported by a grant 
of the Israel Science Foundation (ISF 668/13)}

\thanks{The research of Y. G. was supported by a grant 
of the Israel Science Foundation (ISF 2095/15)}

\maketitle

\section*{Introduction}

For a compact metric space $X$ let $\Homeo(X)$ denote the Polish group
of self homeomorphisms of $X$ equipped with the compact open topology.
In this work a {\em $G$-flow} $(X, G)$ is a pair consisting of a compact (usually metrizable) space $X$
 and a continuous homomorphism of the topological group $G$ into  $\Homeo(X)$.
 We usually write $(g, x) \mapsto gx$ for the action of $G$ on $X$ which is defined via
 this representation. A flow $(X,\Z)$, with the group of integers $\Z$ as the acting group, is called a
{\em cascade} and is usually denoted as $(X, T)$ where $T$ is the homeomorphism
which corresponds to $1 \in \Z$.

For a flow $(X,G)$ let $E(X,G)$ denote its enveloping semigroup.
Recall that $E(X,G)$ is defined as the closure, in the compact space $X^X$,
of the collection of homeomorphisms that is the image of the homomorphism from $G$ 
into $\Homeo(X)$ which defines the dynamical system $(X,G)$.
$E(X,G)$ has a structure of a compact right topological semigroup and it is a $G$-dynamical system
as well. By a theorem of Ellis it always has minimal left ideals, which coincide with its minimal subsystems.
We let ${\rm{mi}}(X,G)$ denote the cardinality
of the collection of minimal left ideals in $E(X,G)$.
For the definition of PI-flows and for more details on enveloping semigroups 
and the structure theory of minimal flows we refer the reader to \cite{Gl-PF}, \cite{V}  
and \cite{Au}. (See also Section \ref{sec1} below.)

The main result of \cite{Gl-75}, actually stated as the title of that paper, is that (for any acting group $G$)
a metric minimal flow whose enveloping semigroup contains finitely many minimal left ideals is PI.
This result was later greatly improved by McMahon \cite{M}, who showed that,  
a metric minimal flow 
with ${\rm{mi}}(X, G)  <  2^{\aleph_1}$ is PI.
In the first section of this note, following a suggestion of Akin,  
we improve McMahon's result and show that 
a metric minimal flow whose enveloping semigroup contains less than 
$2^{\frak{c}}$ (where ${\frak{c}} ={2^{\aleph_0}}$) minimal left ideals is PI.

At the end of \cite{Gl-75} the first named author claimed that a certain metric minimal PI $G$-flow $(X,G)$, 
with $G = SL(2,\R)$,
has an enveloping semigroup $E(X,G)$ with ${\frak{c}}$ minimal left ideals.
Unfortunately the argument given in \cite[Example on page 91]{Gl-75}
does not at all show this. Instead it  
actually shows that each minimal left ideal in $E(X,G)$ contains  ${\frak{c}}$ idempotents.
In Section \ref{sec2} we construct an example of a metric minimal PI flow $(\Om, {\bf{G}})$
(of order $2$), with ${\bf{G}} = G^\N = SL_2(\R)^\N$ and
with the property that ${\rm{mi}}(\Om, {\bf{G}}) = {\frak{c}}$.
Then, in Section \ref{cascadec}, we use this example and results established in \cite{GW-79}
to construct a metric minimal PI cascade $(X,T)$ (of order $3$)
with ${\rm{mi}}(X,T) = {\frak{c}}$.
In Section \ref{cantor} we construct a minimal PI-flow $(X, G)$ (of order $2$) on the Cantor set $X$,
with a suitable acting group $G$ and 
with ${\rm{mi}}(X, G) = 2^{\frak{c}}$.
Finally, in Section \ref{final},
we construct a minimal PI cascade (of order 3)  $(X, T)$
with ${\rm{mi}}(X, T) = 2^{\frak{c}}$.
Thus this final result shows that, even for cascades, the converse 
of the 
implication, 
${\bf{mi}}(X,T)  <  2^\frak{c} \imp PI$,
does not hold.
\footnote{The Morse minimal set $(X,T)$, which is PI of order 3, has ${\rm{mi}}(X,T) = 2$
(see \cite[Theorem 3.1]{HJ-99}).
In Remark \ref{2^n} at the end of Section \ref{cascadec} we indicate how to construct, given $n \in \N$, a
minimal PI cascade with $2^n$ minimal left ideals.
It seems that it should not be too hard to construct minimal PI cascades
with ${\rm{mi}}(X,T) = n$ for any $n =1,2,\dots, \aleph_0$.
However, we leave this question open.}


To sum up, we show in this work that the range of the function ${\bf{mi}}(\cdot)$, on the domain of
minimal metrizable cascades, includes 
 the set of cardinals $\{2^\eta : \eta = 0,1,2,\dots, \aleph_0, \aleph_1, \dots, \mathfrak{c}\}$.
It seems that it should not be too hard to construct minimal PI cascades
with ${\bf{mi}}(X,T) = n$ for any 
$n \in \N$.
However, we leave that question open
\footnote{The Morse minimal set $(X,T)$, which is PI of order 3, has ${\bf{mi}}(X,T) = 2$
(see \cite[Theorem 3.1]{HJ-99}) and \cite{S-17}.}.

We thank Ethan Akin for suggesting the approach taken in the proof of 
Theorem \ref{2c}, and Petra Staynova for pointing out the error in \cite{Gl-75}
and for producing the figure of the equivalent minimal idempotents
$u$ and $u'$.

\br

\section{Non PI-flows have $2^{\frak{c}}$ minimal left ideals}\label{sec1}

\br

The way we will usually control the number of minimal left ideals in an enveloping semigroup $E$
is by computing the number of minimal idempotents in $E$ which are equivalent to a given
minimal idempotent $u$. We recall the definitions and basis facts concerning these notions.
(See e.g. \cite[Chapter 1, Section 2]{Gl-PF}, or \cite[Chapter 1]{Gl-03}.)

\begin{thm}\label{exe-envel}
\begin{enumerate}
\item
A subset $M$ of $E$ is a minimal left
ideal of the semigroup $E$ iff it is a minimal
subsystem of $(E,G)$. In particular a
minimal left ideal is closed.
Minimal left ideals $M$ in $E$ exist and for each
such ideal the set of idempotents in $M$, denoted
by $J=J(M)$, is non-empty.
We say that an idempotent in $E$ is a {\em minimal idempotent}
if it belongs to some minimal left ideal
\item
Let $M$ be a minimal left ideal and $J$ its set
of idempotents then:
\begin{enumerate}
\item[(a)]  For $v\in J$ and $p\in M$, $pv=p$.
\item[(b)] For each $v\in J$,\ $vM=
\{vp:p\in M\}$ is a subgroup of $M$ with
identity element $v$.
For every $w\in J$ the map $p\mapsto wp$
is a group isomorphism of $vM$ onto $wM$.
\item[(c)] $ \{vM:v\in J\}$ is a partition
of $M$. Thus if $p\in M$ then there exists a unique
$v\in J$ such that $p\in vM$;\ we denote
by $p^{-1}$ the inverse of $p$ in $vM$.
\end{enumerate}
\item
Let $K,L,$ and $M$ be minimal left ideals of $E$.
Let $v$ be an idempotent in $M$, then there exists
a unique idempotent $v'$ in $L$ such that
$vv'=v'$ and $v'v=v$. (We write $v\sim v'$
and say that $v'$ is {\em equivalent\/} to $v$.)
If $v''\in K$ is equivalent to $v'$, then
$v''\sim v$. The map $p\mapsto pv'$ of
$M$ to $L$ is an isomorphism of $G$-systems.
\end{enumerate}
\end{thm}

\br

Now to PI-flows.
We say that a minimal flow $(X, G)$  is a
{\em strictly PI flow} if there is an ordinal $\eta$
(which is countable when $X$ is metrizable)
and a family of flows
$\{(W_\iota,w_\iota)\}_{\iota\le\eta}$
such that (i) $W_0$ is the trivial flow,
(ii) for every $\iota<\eta$ there exists a homomorphism
$\phi_\iota:W_{\iota+1}\to W_\iota$ which is
either proximal or equicontinuous
(isometric when $X$ is metrizable), (iii) for a
limit ordinal $\nu\le\eta$ the flow $W_\nu$
is the inverse limit of the flows
$\{W_\iota\}_{\iota<\nu}$,  and
(iv) $W_\eta=X$.
We say that $(X, G)$ is a {\em PI-flow} if there
exists a strictly PI flow $\tilde X$ and a
proximal homomorphism $\theta:\tilde X\to X$.


We say that the extension $(X, G) \overset{\pi}{\to} (Y,G)$ is a RIC ({\em relatively incontractible})
{\em extension} if 
for every $n \in \N$ the minimal points are dense in the relation
$$
R^{(n)}_\pi =\{(x_1,x_2,\dots,x_n) \in X^n : \pi(x_i) = \pi(x_j) \ \forall \ 1 \le i, j \le n\}.
$$
It is easy to see that a RIC extension satisfies the {\em generalized Bronstein condition}
as defined in \cite[page 813]{V}, with $\pi_i = \pi, \ \forall \ 1 \le i \le n$. 
Also a RIC extension is necessarily an open map. Clearly every distal extension is RIC, so
in particular, every distal extension is open.

The structure theorem for the general minimal flow is proved in
\cite{EGS} and \cite{McM} (see also \cite{V}) and asserts that
every minimal flow admits a canonically defined proximal
extension which is a weakly mixing RIC extension of a strictly PI flow.
Both the Furstenberg and the Veech-Ellis structure theorems 
(for minimal distal and point-distal flows respectively) are
corollaries of this general structure theorem.
For more details on the structure theorem see \cite{Gl-PF}, \cite{V}
and \cite{Au}.

\begin{thm}[Structure theorem for minimal flows]\label{structure}
Given a minimal flow $(X, G)$, there exists an ordinal $\eta$
(countable when $X$ is metrizable) and a canonically defined
commutative diagram (the canonical PI-Tower)
\begin{equation*}
\xymatrix
        {X \ar[d]_{\pi}             &
     X_0 \ar[l]_{{\theta}^*_0}
         \ar[d]_{\pi_0}
         \ar[dr]^{\sigma_1}         & &
     X_1 \ar[ll]_{{\theta}^*_1}
         \ar[d]_{\pi_1}
         \ar@{}[r]|{\cdots}         &
     X_{\nu}
         \ar[d]_{\pi_{\nu}}
         \ar[dr]^{\sigma_{\nu+1}}       & &
     X_{\nu+1}
         \ar[d]_{\pi_{\nu+1}}
         \ar[ll]_{{\theta}^*_{\nu+1}}
         \ar@{}[r]|{\cdots}         &
     X_{\eta}=X_{\infty}
         \ar[d]_{\pi_{\infty}}          \\
        pt                  &
     Y_0 \ar[l]^{\theta_0}          &
     Z_1 \ar[l]^{\rho_1}            &
     Y_1 \ar[l]^{\theta_1}
         \ar@{}[r]|{\cdots}         &
     Y_{\nu}                &
     Z_{\nu+1}
         \ar[l]^{\rho_{\nu+1}}          &
     Y_{\nu+1}
         \ar[l]^{\theta_{\nu+1}}
         \ar@{}[r]|{\cdots}         &
     Y_{\eta}=Y_{\infty}
    }
\end{equation*}
where for each $\nu\le\eta, \pi_{\nu}$
is RIC, $\rho_{\nu}$ is isometric, $\theta_{\nu},
{\theta}^*_{\nu}$ are proximal and
$\pi_{\infty}$ is RIC and weakly mixing {\bf{of all orders}}.
For a limit ordinal
$\nu ,\  X_{\nu}, Y_{\nu}, \pi_{\nu}$
etc. are the inverse limits (or joins) of
$ X_{\iota}, Y_{\iota}, \pi_{\iota}$ etc. for $\iota
< \nu$.
Thus $X_\infty$ is a proximal extension of $X$ and a RIC
weakly mixing extension of the strictly PI-flow $Y_\infty$.
The homomorphism $\pi_\infty$ is an isomorphism (so that
$X_\infty=Y_\infty$) iff $X$ is a PI-flow.
\end{thm}

The weak mixing of all orders of the extension $\pi_\infty$ means that for
every $n \in \N$ the relation
$$
R_{\pi_\infty}^{(n)} =
\{(x_1,x_2, \dots, x_n) \in X_\infty^n : \pi_\infty(x_i) = \pi_\infty(x_j), \ \forall \ 1 \le i, j \le n\}
$$
is topologically transitive. This augmented version of the structure theorem follows from
Theorem 2.6.2 in Veech's review paper \cite{V}. It can also be proven by a slight modification 
of E. Glasner's proof of the structure theorem as it appears in Theorem 27, Chapter 14 (page 219) in \cite{Au}.

\br

Let now $(X, G)$ be a minimal metric flow and $\pi : (X, G) \to (Y, G)$ a homomorphism.
Let $\Xi$ denote the Cantor set. We choose a sequence $\mathcal{U}_n$ 
of clopen 
partitions 
of $\Xi$ so that $\mathcal{U}_{n +1}
\prec \mathcal{U}_n$ and so that $\bigcup_{n \in \N} \mathcal{U}_n$
generates the topology on $\Xi$. We let $k_n = \card \mathcal{U}_n$.

Let $R_\pi^\Xi$ be the compact subset of $X^\Xi$ defined by the condition
$$
R_\pi^\Xi = \{r \in X^\Xi : \pi(r(\xi_1)) = \pi(r(\xi_2)), \ \forall \ \xi_1, \xi_2 \in \Xi\}. 
$$
That is, the set of $r$ such that  $\pi \circ r : \Xi \to Y$ is a constant function. 
We denote by $CR_\pi^\Xi$ the subset of $R_\pi^\Xi $ consisting of continuous maps,
and by $UCR_\pi^\Xi$ the topological space whose underlying set is $CR_\pi^\Xi$
but equipped with the topology of uniform convergence; i.e. the topology induced by the
metric
$$
D(r_1, r_2) = \sup_{\xi \in \Xi} d(r_1(\xi), r_2(\xi)). 
$$
It is easy to see that $D$ is a complete metric.
For each $n \in \N$ 
we will identify the space $R_\pi^{k_n}$ with the subspace of $UCR_\pi^\Xi$ consisting of
functions which are constant on each cell of $\mathcal{U}_n$.
Clearly $\bigcup_{n \in \N} R_\pi^{k_n}$ is a dense subset of both spaces
$UCR_\pi^\Xi$ and $R_\pi^\Xi$. In particular 
it follows that $UCR_\pi^\Xi$ is a Polish space.

Let $\iota : UCR_\pi^\Xi  \to CR_\pi^\Xi \subset R_\pi^\Xi$ denote the natural map.
Clearly $\iota$ is a continuous injection. We observe that the image space $CR_\pi^\Xi$ is dense 
in $R_\pi^\Xi$.
Of course all of these spaces are $G$-invariant; however, whereas the spaces
$R_\pi^{k_n}$ and $R_\pi^\Xi$ are indeed flows, in the sense that their phase spaces are compact,
the space $UCR_\pi^\Xi$ is merely a Polish space and thus $(UCR_\pi^\Xi, G)$ is a {\em Polish
dynamical system}. We now have the following crucial proposition whose proof mimics that of
\cite[Proposition 4.10]{A-97}.

\begin{prop}\label{Ethan}
Let $(X, G)$ be a minimal metric flow and $(X, G) \overset{\pi}{\to} (Y,G)$ a homomorphism.
Suppose $\pi$ is a weakly mixing extension of all (finite) orders then 
the (non-metrizable) flow $(R_\pi^\Xi, G)$ is topologically transitive and has a dense set of 
transitive points.
\end{prop}

\begin{proof}
We will first show that the Polish dynamical system $(UCR_\pi^\Xi, G)$ is topologically
transitive. So let $V_1, V_2$ be two nonempty open subsets of $UCR_\pi^\Xi$.
There are $n_1, n_2 \in \N$ with $V_i \cap R_\pi^{k_{n_i}} \not=\emptyset, \ i=1,2$.
Then, with  $n = \max(n_1, n_2)$ we have $V_i \cap R_\pi^{k_n} \not=\emptyset, \ i=1,2$.
Since, by assumption, the flow $(R_\pi^{k_n}, G)$ is topologically transitive there is some
$g \in G$ with
$$
\emptyset \not = g (V_1 \cap R_\pi^{k_n}) \cap (V_2 \cap R_\pi^{k_n}) 
\subset g V_1  \cap V_2.   
$$
Thus the Polish system $(UCR_\pi^\Xi, G)$ is indeed topologically transitive,
and being Polish, this fact implies that it has a dense $G_\del$ subset of transitive points.
Now under $\iota$ this set is pushed onto a dense set of transitive points in the flow
$(R_\pi^\Xi,G)$ and our proof is complete.
\end{proof}

We are now ready to state and prove the main result of this section, Theorem \ref{2c} below.
Our proof of this theorem follows the main ideas of McMahon's proof in \cite{M}.
The tools that enable us to improve his result are the augmented form of the structure theorem,
Theorem \ref{structure}, and Proposition \ref{Ethan}.

\begin{thm}\label{2c}
Let $(X, G)$ be a minimal metric flow which is not PI. Then
$E(X, G)$ has $2^{\frak{c}}$ minimal left ideals.
\end{thm}

\begin{proof}
We start by recalling the following proposition from \cite{M}.

\begin{prop}\label{11}
 If the extension $\pi : (X, G) \to (Y, G)$ is proximal, then $\pi$ induces a one-to-one
 correspondence between the collections of left minimal ideals in  $E(X, G)$ and $E(Y, G)$.
\end{prop} 
 
\begin{proof}
Let $\pi_*$ be the induced semigroup homomorphism from  $E(X, G)$ onto $E(Y, G)$. 
Suppose $I_1, I_2$ are two distinct (hence disjoint) minimal ideals in  $E(X, G)$ 
with $I = \pi_*(I_1) = \pi_*(I_2)$.
Let $u_1 \in I_1$ and $u_2 \in I_2$ be equivalent idempotents, so that $u_1u_2 = u_2$
and $u_2u_1 = u_1$.
Then, as $\pi_*(u_1)$ and $\pi_*(u_2)$ belong to the same minimal left ideal $I$ in $E(Y, G)$,
we have $\pi_*(u_1) = \pi_*(u_2u_1) = \pi_*(u_2)\pi_*(u_1) = \pi_*(u_2)$.
So for every $x \in X$,
$$
\pi(u_1x) = \pi_*(u_1)\pi(x) = \pi_*(u_2) \pi(x) = \pi(u_2x).
$$
Thus the points $u_1x$ and $u_2x$ are proximal.
But, since we have both $u_1(u_1x) = u_1x$ and $u_1(u_2x) = u_2x$,
these points are also distal, whence $u_1x = u_2x$.
This means that $u_1 =u_2 \in I_1 \cap I_2$, a contradiction.
\end{proof}

Now, in the PI diagram constructed for $(X, G)$ in the structure theorem \ref{structure},
the extension $X_\infty \to X$ is a proximal extension. Hence, in view of Proposition \ref{11},
in counting the minimal left ideals we can replace $X$ by $X_\infty$.
So, from now on we assume that $X = X_\infty$ and that the extension
$X = X_\infty \to Y = Y_\infty$ is, nontrivial, RIC and weakly mixing of all orders.

Applying Proposition \ref{Ethan} we conclude that the flow
$(R_\pi^\Xi, G)$ has a dense set of transitive points.
Recall that the enveloping semigroup $E(X, G)$ can be identified with the 
enveloping semigroup $E(R_\pi^\Xi, G)$ (in fact also with $E(X^\Xi, G)$)
by letting $p \in E(X, G)$ act on  $X^\Xi$ coordinaetwise.

Let $r \in R_\pi^\Xi$ be a fixed transitive point. Let $y$ be the unique point in $Y$
such that $r(\Xi) \subset \pi^{-1}(y)$. We also fix a left minimal ideal $I$ in $E(X, G)$.
Next observe that the extension $\pi$, being nontrivial and RIC, can not be proximal.
It follows that there is a minimal idempotent $u \in I$ with $\card{(u\pi^{-1}(y))} >1$.
We fix two distinct points $x$ and $x'$ in $u\pi^{-1}(y)$.

For each nonempty subset $A \subset \Xi$
let $r_A \in R_\pi^\Xi$ be defined by
$$
r_A(\xi) = \begin{cases} x & \xi \in A  \\ 
x' & \xi \in A^c
\end{cases}
$$
Since $r$ is a transitive point there is an element $q_A \in E(X, G)$ with $q_A r =r_A$.
Let $I_A = Iq_A$. This is a minimal left ideal in $E(X, G)$.
Let $u_A$ be the unique minimal idempotent in $I_A$ which is equivalent to $u$.
Then for some element $s_A \in I_A$ we have  $u_A = s_A q_A$.
Suppose $A, A' $ are two distinct nonempty subsets of $\Xi$. 
Then $u_A r (\xi) = s_A q_A r(\xi) = s_A r_A (\xi)$ which equals $s_Ax$ or  $s_Ax'$,
according to whether $\xi \in A$ or $\xi \in A^c$.
Similarly, $u_{A'} r (\xi) = s_{A'} q_{A'} r(\xi) = s_{A'} r_{A'} (\xi)$ which equals $s_{A'}x$ or  $s_{A'}x'$,
according to whether $\xi \in A'$ or $\xi \in A'^c$.
Since $x$ and $x'$ form a pair of distinct
points in $uX$ ($ux =x, \ ux'=x'$)
we conclude that
$s_Ax \not= s_Ax'$ and $s_{A'}x \not = s_{A'}x'$. 
(E.g., if $s_Ax = s_Ax'$ then also
$x = u_A x = u_A s_A^{-1}s_Ax = u_A s_A^{-1}s_Ax' = u_Ax' = x'$, a contradiction.)
It follows that
$u_A r \not = u_A' r$, hence $u_A \not= u_{A'}$.
Thus $I_A \not = I_{A'}$ and the map $A \mapsto I_A$ is one-to-one.
This shows that $E(X, G)$ has at least $2^{\frak{c}}$ minimal ideals. 
By cardinality arguments ${\rm{mi}}(X, G) \le 2^{\frak{c}}$ and the proof of Theorem
\ref{2c} is complete.


\end{proof}

\br
\newpage

\section{A metric minimal PI flow with $\frak{c}$ minimal left ideals}\label{sec2}

Let $X \cong S^1$ and $Y \cong \mathbb{P}^1$ be as in \cite{Gl-75}, 
namely the spaces of rays emanating from the origin, and lines
through the origin of $\R^2$, respectively. Let $G=SL_2(\R)$ and we consider $(X,G)$ and $(Y,G)$
as $G$-flows. Then, the map $\pi : X \to Y$, which sends a ray into the unique line in which it is contained, is
a homomorphism of minimal $G$-flows. It is easy to see that $Y$ is proximal and that the extension $\pi$ 
is a group extension with fiber group $\{\rho, \id\} \cong \Z_2$, where $\rho : X \to X$ is the
map which sends a ray onto its antipodal ray. Thus the $G$-flow $(X,G)$ is in particular a PI-flow.


As was shown in \cite[Example on page 91]{Gl-75}
each minimal left ideal in $E(X,G)$ contains  ${\frak{c}}$ idempotents, called {\em minimal idempotents}.
Any minimal idempotent $u$ corresponds to a partition of the circle $S^1$ (naturally identified with $X$)
into two arcs, say $J_1$ and $J_2$ with common (antipodal) end points $x_1$ and $x_2$, which are
fixed by $u$, and such that $u(J_1) = \{x_1\}, u(J_2) =\{x_2\}$.
For example, the figure below describes two minimal equivalent idempotents 
$u$ and $u'$ (thus belonging to two different minimal left ideals), 
which correspond to the arcs with end points $N, S$:
\br

\begin{figure}[h]
 \centering 
\caption{The minimal idempotents $u$ and $u'$}
\includegraphics[scale=0.6]{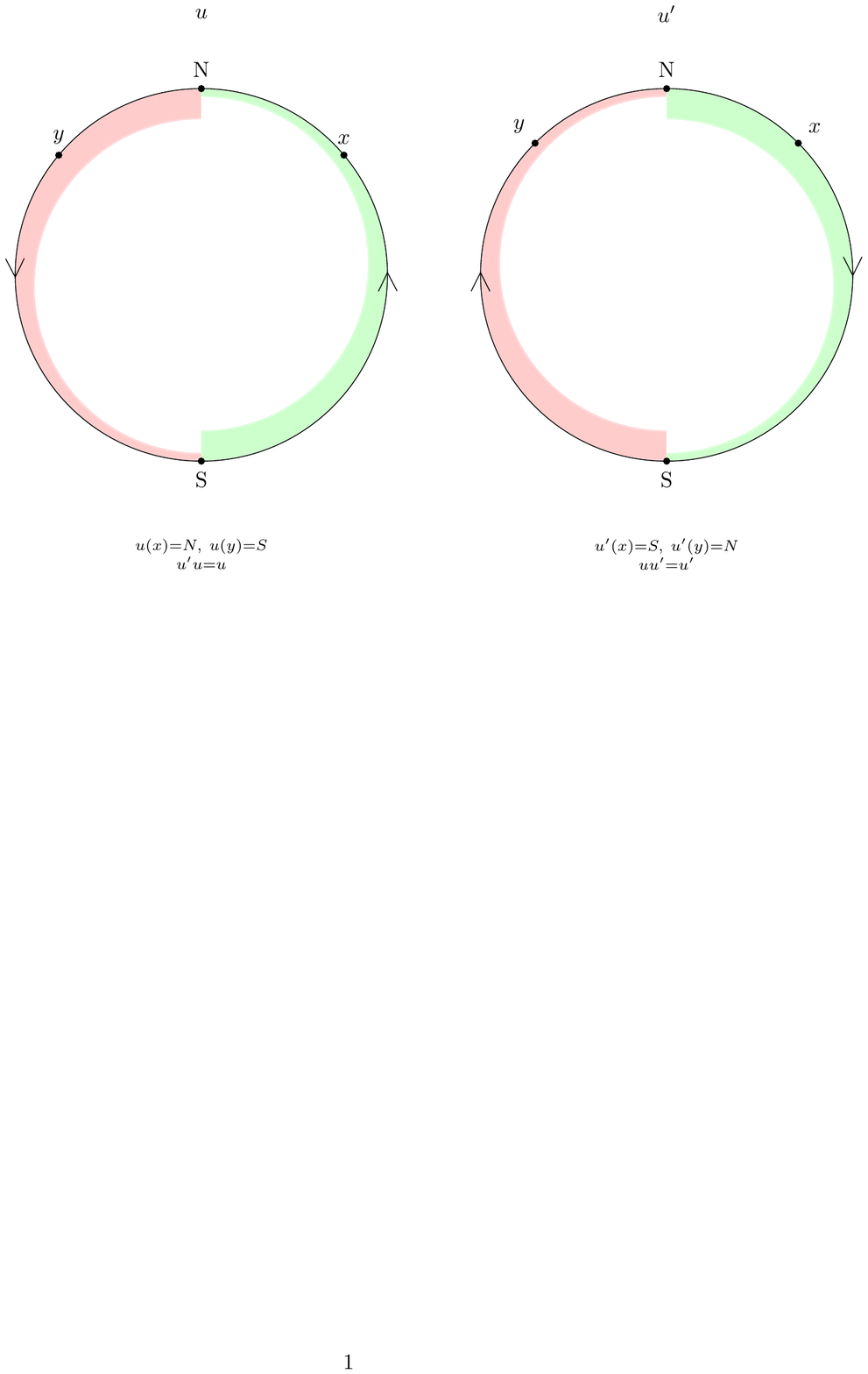}
\label{mi}
\end{figure}
\newpage



%
%
%
%
%
%

In fact, it is not hard to see that $E(X,G)$ contains exactly two minimal left ideals.
For the proof of the following lemma see Appendix \ref{Ap} below.

\begin{lem}\label{idemp}
Let $E = E(X,G)$ be the enveloping semigroup of $(X,G)$. 
For any line $y \in Y$, with corresponding pair of antipodal rays $x_1, x_2$ in $X$,
let $u_y$ and $u'_y$ be the maps from $X$ into itself
which send the corresponding arcs  
onto the endpoints $x_1, x_2$ in the counter clock-wise direction and in the clock-wise direction, respectively.
\begin{enumerate}
\item
$G$ embeds topologically into $E$.
\item
For each $y \in Y$ the maps $u_y, u'_y$ are idempotents in $E$.
\item
The sets $I_+ =\{u_y : y \in Y\}$ and $I_- =\{u'_y : y \in Y\}$ are the two minimal left ideals in $E$.
\item
The set $\{e\} \cup I_+ \cup I_-$ comprises all the idempotents in $E$.
\end{enumerate}
\end{lem}

Next we will describe a construction of a flow $(\Om, \bf{G})$, for a suitable
group $\bf{G}$, which is metric, minimal, PI, and its enveloping semigroup has ${\frak{c}}$
minimal left ideals.

We let $\Om = X^\N$, where $X$ is as above. The group $\bf{G}$ is
the product group ${\bf{G}} = G^\N = SL_2(\R)^\N$.
The action of $\bf{G}$ on $\Om$ is coordinatewise.
For each $n \in \N$ the projection $P_n : \Om \to X$ is a homomorphism of the flow $(\Om, {\bf{G}})$
onto the flow $(X, G)$.

As $\bf{G}$ acts coordinatewise, we can achieve, for every sequence $\theta \in \{1, -1\}^\N$,
a minimal idempotent ${\bf{u}}_\theta   \in E(\Om, {\bf{G}})$ such that $P_n({\bf{u}}_\theta)$ is either $u$
or $u'$, as in the example above, according to whether $\theta(n) = 1$ or $-1$.   

Ranging over the elements $\theta \in  \{1, -1\}^\N$ we obtain this way a collection of
${\frak{c}}$ pairwise equivalent distinct minimal idempotents, which in turn,
belong to  ${\frak{c}}$ distinct minimal left ideals.

Finally, it is not hard to see that $(\Om, {\bf{G}})$ is minimal and PI. In fact,
the natural map $\Pi  = (\pi \circ P_1, \pi \circ P_2, \dots ): \Om \to \Om_0 = Y^\N$ is a group
extension, with compact fiber group $K \cong \Z_2^\N$, and the factor
flow $(\Om_0, \bf{G})$ is proximal.

\br

\section{A cascade with $\frak{c}$ minimal left ideals}\label{cascadec}
We will next apply the construction of skew-product minimal flows from
\cite{GW-79} and the above example
to obtain a minimal PI cascade with 
${\frak{c}}$  minimal left ideals.

In the sequel we use the notations of \cite{GW-79}.
In particular $(Z, \sig)$ is a fixed minimal metric cascade and 
$(Y, \mathcal{G})$ is a minimal $\mathcal{G}$-flow with $\mathcal{G}$
a path-wise connected topological group.
In addition, we assume that $\pi : (Y, \mathcal{G}) \to (Y_0, \mathcal{G})$ is a 
group extension with compact fiber group $K$.
That is, there is a compact subgroup $K \subset \Homeo(Y)$,
the group of homeomorphisms of the compact metric space $Y$,
such that $gky = kgy$ for every $y \in Y, g\in \mathcal{G}$ and $k \in K$,
and such that  $\pi(y_1) = \pi(y_2)$ iff there is a $k \in K$ with $y_2 = ky_1$.
Note that any homeomorphism $h$ of $Y$ which commutes with $K$
defines a unique homeomorphism $\tilde{h}$ of $Y_0$.

As in \cite{GW-79}, on the product space $X = Z \times Y$,
we let $\mathcal{H}_\mathcal{G}(X)$ be the subgroup of $\Homeo(X)$
which consists of homeomorphisms $g : X \to X$
of the form $g(z, y) = (z, g_z(y))$,
where $z \mapsto g_z \in \mathcal{G}$ is a continuous map from $Z$ into $\mathcal{G}$.
We then let 
$$
\mathcal{S}_\mathcal{G}(\sig) =
\{g^{-1} \circ \sig \circ g : g \in \mathcal{H}_\mathcal{G}(X)\}
$$
(here $\sig$ is identified with $\sig \times \id$, where $\id$ is the identity map on $Y$,
so that $g(z,y) = (\sig z, g_{\sig z}^{-1}g_z(y))$).
We set $X_0 = Z \times Y_0$ and define the map $\pi_X : X \to X_0$
as $\pi_X(z,y) = (z, \pi(y))$. 

We note that every homeomorphism $T \in  \ol{\mathcal{S}_\mathcal{G}(\sig)}$
(where the closure is taken in the Polish group $\mathcal{H}(X)$ equipped with
its uniform convergence topology) commutes with the $K$-action on the $Y$ coordinate of $X$. 
Thus, every such $T$ defines also a corresponding homeomorphism $\tilde{T}$ in
$\mathcal{H}(X_0)$. The proof of the following lemma is straightforward and we skip it.

\br

\begin{lem}\label{skew}
For every $T  \in  \ol{\mathcal{S}_\mathcal{G}(\sig)}$, each element $p$ in the enveloping semigroup 
$E(X,T)$ has the form 
$$
p(z,y)  = (\pi_*(p) z, p_zy),
$$
where $\pi_*(p)$ is the image of $p$ in $E(Z,\sig)$ and $p_z$ is an element of $E(Y,\mathcal{G})$.
Moreover $p$ is an idempotent iff $\pi_*(p)$ and all the $p_z, \ z \in Z$, are idempotents. 
\end{lem}

We will assume that the action of $\mathcal{G}$ on $Y_0$ has the following property $(*)$:
\begin{quote}
For every pair of points $y_1, y_2$ in $Y_0$ there exist neighborhoods $U$ and $V$
of $y_1$ and $y_2$ respectively, such that for every 
$\ep >0$ there exists $h \in \mathcal{G}$ with $\diam (h (V \cup U)) < \ep$.
\end{quote}

\br

\begin{prop}\label{prop}\label{residual}
Under the above assumptions
there exists a dense $G_\del$ subset $\mathcal{R} \subset \ol{\mathcal{S}_\mathcal{G}(\sig)}$
such that for every $T \in \mathcal{R}$ we have:
\begin{enumerate}
\item
The cascade $(X, T)$ is minimal.
\item
The extension $\pi_X : (X, T)  \to (X_0,\tilde{T})$ is a $K$-extension.
\item
The projection $P :  (X_0, \tilde{T}) \to (Z, \sig)$ is a proximal extension.
\end{enumerate}
In particular, for $T \in \mathcal{R}$ the minimal  cascade $(X,T)$ is PI.
\end{prop}

\begin{proof}
Part (1) is just Theorem 1 in \cite{GW-79}.
Part (2) is clear. For part (3) apply the proof of Theorem 3 in \cite{GW-79} with the following modification.
In the definition of the sets $E_{U, V, \ep}$ we consider only nonempty open subsets
$U$ and $V$ of $Y$ that are $K$-invariant, and the diameter is taken with respect to the
pseudometric on $X$ which is obtained by lifting the $X_0$ metric via $\pi_X$.
Thus for this kind of open sets and with respect to this pseudometric, the
conditions required in the proof of Theorem 3 in \cite{GW-79}
are satisfied by our assumption on the $\mathcal{G}$-action on $Y_0$.
\end{proof}

\br

\begin{thm}\label{cas}
There exists a metric minimal PI cascade (of order $3$) whose enveloping semigroup has 
${\frak{c}}$ minimal left ideals.
\end{thm}

\begin{proof}
We are going to apply Proposition \ref{prop} with the following input.
$(Z,\sig)$ is a fixed metric minimal equicontinuous cascade (e.g. one can take $(Z,\sig)$ to be
an irrational rotation on the unit circle in $\C$).
The role of $Y$ and $Y_0$ will be played by $\Om$ and $\Om_0$ respectively, 
with $\mathcal{G} = {\bf{G}} = SL_2(\R)^\N$,
as described in Section \ref{sec2}.
Clearly property $(*)$ holds for the action of ${\bf{G}}$ on $\Om_0$.
Therefore, Proposition \ref{prop} applies and we pick any $T$ in the residual set 
$\mathcal{R}$ provided by this proposition. We also fix an arbitrary point $z_0 \in Z$.

Let now $u$ be any minimal idempotent in $E(X,T)$. 
By the equicontinuity of $(Z, \sig)$,  $u$ projects to $Z$ as the identity.
By Lemma \ref{skew},
for any $\om \in \Om$, $u(z_0, \om) = (uz_0, u_{z_0}\om) = (z_0, u_{z_0}\om)$, where $u_{z_0}$
is an idempotent in $E(\Om, {\bf{G}})$.
Next observe that an element $p$ of $E(\Om, \mathcal{G})$ is completely determined by its
projections $P_n(p), \ n \in \N$, in $E(S^1, SL_2(\R))$. Of course for each $n \in \N$,
$P_n(u_{z_0})$ is an idempotent in $E(X, G)$ hence, by Lemma \ref{idemp} (4),
either $P_n(u_{z_0})= u_{P_n(\om)}$, or  $P_n(u_{z_0})= u'_{P_n(\om)}$
(by Lemma \ref{idemp} (1),  $P_n(u_{z_0})\not = e$).
One more observation: the reflection $\rho$ (which sends a ray onto its antipodal ray)
satisfies $\rho(u_{z_0}) = u'_{z_0}$ and it follows that under the action of $K \cong \{\rho, \id\}^\N$
the orbit $Ku$, restricted to the fiber over $z_0$, namely $Ku_{z_0}$
exhausts all the possible  ${\bf{u}}_\theta, \ \theta \in \{1, -1\}^\N$.

We finally conclude, as in the discussion in Section \ref{sec2}, that $E(X,T)$ indeed contains
${\frak{c}}$ minimal left ideals.
\end{proof}

\begin{rmk}\label{2^n}
For any $n \in \N$ the same method,
where in Section \ref{sec2} we set $\Om = X^n$ instead of $X ^{\N}$,
will produce, a minimal cascade with $2^n$ minimal left ideals. 
\end{rmk}

\br

\section{A PI-flow on the Cantor set with $2^{{\frak{c}}}$ minimal left ideals}\label{cantor}

Let $Z$ denote the Cantor set. Let $G_0 \subset \Homeo(Z)$ be a finitely generated
subgroup such that the corresponding flow $(Z, G_0)$ is {\em totally minimal}
in the sense that for every two ordered finite subsets of distinct points, $\{z_1,z_2, \dots, z_n\}$ and
$\{w_1,w_2, \dots, w_n\}$ of $X$, there is a sequence of elements $\ga_j \in G_0$ with 
$\lim_{j\to \infty} \ga_j(z_i) = w_i$, for $i=1,2,\dots,n$. 
\footnote{Actually a less stringent condition will suffice:
 for every two ordered finite subsets of distinct points, $\{z_1,z_2, \dots, z_n\}$ and
$\{w_1,w_2, \dots, w_n\}$ of $X$ and every pair of disjoint nonempty open sets
$V_1, V_2 \subset X$ there is an element $g \in G_0$ with
$\ga z_i \in V_1$ and $\g w_i \in V_2$ $\forall \ 1 \le i \le n$.}

Clearly then the flow $(Z, G_0)$ is,
in particular, minimal and proximal.
We note that the existence of $G_0$ as above follows from a result of Kechris and Rosendal
\cite[Theorem 2.10]{KR}, according to which the Polish group $\Homeo(Z)$ is topologically $2$-generated; 
i.e. it contains a dense subgroup which is generated by two elements. 

Let $\Z_2 = \{1,-1\}$ and set $X = Z \times \Z_2$. Choose a fixed nonempty clopen proper subset $U$ of $Z$.
Let $\phi : Z \to \Z_2$ be the function 
$$ 
\phi(z) = \begin{cases} -1 & z \in U\\  1 & z \in U^c  \end{cases}
$$
and define a homeomorphism $\Phi$ of $X$ by the formula
$$
\Phi(z,\ep) = (z, \phi(z)\ep).
$$
By abuse of notation, given $\ga \in G_0$, we also denote by the same letter $\ga$ the homeomorphism 
$\ga \times \id$ of $X$. Thus $\ga (z, \ep) = (\ga z, \ep)$.
Finally let $G$ denote the subgroup of $\Homeo(X)$ generated by $G_0$ (acting on $X$) and $\Phi$.

\begin{prop}
The flow $(X, G)$ is PI and its enveloping semigroup has $2^{\frak{c}}$ minimal left ideals. 
\end{prop}

\begin{proof}
By construction $(X, G)$ is a group extension (with fiber group $\Z_2$) of a proximal flow.
Thus it is a strictly PI flow of order $2$. It is easy to check that it is minimal.
Fix an arbitrary point $z_0 \in Z$ and let $\{V_n\}_{n \in \N}$ be a nested sequence of clopen neighborhoods of 
$z_0$ with $\bigcap V_n = \{z_0\}$.

Given, any two nonempty disjoint finite sets $E, F \subset Z$ and an $n \in \N$
choose elements $\al = \al_{E, F, n}$ and $ \beta = \beta_{E, F, n}$  in $G_0$ such that
(i) $\al(E) \subset U$, (ii) $\al(F) \subset U^c$ and (iii) $\beta(\al(E) \cup \al(F)) \subset V_n$. 
Then form the product 
$$\ga_{E, F, n} = \beta \Phi \al.
$$

Given an infinite subset $A \subset Z$ with an infinite complement $B = Z \setminus A$
we consider the directed set $\{(E, F, n) : E \subset A, F\subset B, n \in \N \}$,
where $(E , F, n) \prec (E' , F', n')$
iff $E \subset E', F \subset F'$ and $n \le n'$.
A moment's reflection will convince the reader that 
for $A \subset Z$ such that $z_0 \not\in A$,
the net $\{\ga_{E, F, n}\}\subset G$ 
converges to a minimal idempotent
$u_A \in E(X, G)$ defined by the formula
$$
u_A(z,\ep) =  \begin{cases} (z_0, -\ep) & z \in A\\  (z_0,\ep) & z \in B  \end{cases}
$$
Clearly $u_A u_{A'} = u_{A'}$ for any choice of sets $A, A' \subset Z$ as above,
i.e. $u_A$ and  $u_{A'}$ are equivalent minimal idempotents,
and it follows that $E(X, G)$  has at least $2^{\frak{c}}$ minimal left ideals.
By cardinality arguments ${\rm{mi}}(X, G) \le 2^{\frak{c}}$ and our proof is complete. 
\end{proof}

\br

\section{A cascade with $2^\frak{c}$ minimal left ideals}\label{final}

In this last section we show how to construct a minimal PI cascade with $2^\frak{c}$ minimal left ideals.
The construction is similar to the constructions described in Sections \ref{cascadec} and \ref{cantor}.

Now, back to the context of Section \ref{cascadec}.
Let $Y_0 = S^2$, the $2$-sphere in $\R^3$, 
and let $\mathcal{G}_0 = \Homeo_0(S^2)$ denote the connected component of the identity of 
the Polish group $\Homeo(S^2)$.
We observe that $\mathcal{G}_0 =\Homeo_0(S^2)$
acts totally minimally and extremely proximaly on $S^2$.
Let $K = S^1 = \{\zeta \in \C : |\zeta|=1\}$ be the circle group and set $Y = S^2 \times S^1 = Y_0 \times K$. 


Choose two open sets $U_1, U_2 \subset Y_0$ with
$\ol{U_1} \cap \ol{U_2} = \emptyset$.
Let $\phi : Y_0 \to [0,1]$ be a continuous function such that
$$
\phi(y) = \begin{cases} 1  & y \in U_1\\  0 & y \in U_2  \end{cases}
$$
Now define one parameter family of homeomorphism $\Phi_t$ of $Y$ by the formula
$$
\Phi_t(y, \zeta) = (y, e^{i\pi t \phi(y)}\zeta).
$$
We set $\Phi=\Phi_1$, so that $\Phi(y,\zeta) = (y,\zeta)$ for $\zeta \in U_2$ and $(y,-\zeta)$ for $\zeta \in U_1$. 
By abuse of notation, given $g \in \mathcal{G}_0$, we also denote by the same letter $g$ the homeomorphism 
$g \times \id$ of $Y = Y_0 \times S^1$. Thus $g (y, \zeta) = (g y, \zeta)$.
Denote by $\mathcal{G}$ the subgroup of $\Homeo(Y)$ generated by the following three 
path-wise connected subgroups: (i) $\mathcal{G}_0$ (acting on $Y$), (ii) $\{\Phi_t \ | \ t \in \R \}$ and (iii) $\id \times K$. 
Clearly $\mathcal{G}$ is itself path-wise connected. 
Also one can easily check that $K$ is central in $\mathcal{G}$.

%

\begin{prop}\label{claim2}
The flow $(Y, \mathcal{G})$ is PI and its enveloping semigroup has $2^{\frak{c}}$ minimal left ideals. 
\end{prop}

\begin{proof}
By construction $(Y, \mathcal{G})$ is a group extension (with fiber group $S^1$) 
of the extremely proximal flow $(Y_0, \mathcal{G}_0)$.
Thus it is a strictly PI flow of order $2$. It is clearly minimal.
Fix an arbitrary point $y_0 \in Y_0$ and let $\{V_n\}_{n \in \N}$ 
be a nested sequence of clopen neighborhoods of 
$y_0$ with $\bigcap V_n = \{y_0\}$.

Given two nonempty disjoint finite sets $E, F \subset Y_0$ and $n \in \N$
choose elements $\al = \al_{E, F, n}$ and $ \beta = \beta_{E, F, n}$  in $\mathcal{G}_0$ such that
(i) $\al(E) \subset U_1$, (ii) $\al(F) \subset U_2$ and (iii) $\beta(\al(E) \cup \al(F)) \subset V_n$. 
Then form the product 
$$
g_{E, F, n} = \beta \Phi \al.
$$

Given an infinite subset $A \subset Y_0$ with an infinite complement $B = Y_0 \setminus A$
we consider the directed set $\{(E, F, n) : E \subset A, F\subset B \ {\text{finite}}, \  n \in \N \}$,
where $(E , F, n) \prec (E' , F', n')$
iff $E \subset E', F \subset F'$ and $n \le n'$.

It is now easy to see that, for $A$ with $y_0 \not\in A$, 
the net $\{g_{E, F, n}\}\subset \mathcal{G}$ 
converges to the minimal idempotent
$u_A \in E(Y, \mathcal{G})$ defined by the formula
$$
u_A(y, \zeta) =  \begin{cases} (y_0, -\zeta) & y \in A\\  (y_0, \zeta) & y \in B  \end{cases}
$$
(Clearly $u_A$ is an idempotent. To see that it is minimal observe that for $\ga \in \mathcal{G}_0$
we have
$\ga u_A (y, \zeta) = (\ga y_0, \pm \zeta).$
)

It is easy to see that $u_A u_{A'} = u_{A'}$ for any choice of sets $A, A' \subset Y_0$ as above,
i.e. $u_A$ and  $u_{A'}$ are equivalent minimal idempotents.
It follows that $E(Y,\mathcal{G})$ has at least $2^{\frak{c}}$ minimal left ideals.
By cardinality arguments ${\rm{mi}}(Y, \mathcal{G}) \le 2^{\frak{c}}$ and our proof is complete. 
\end{proof}

\br

In \cite{D} M. Dirb\'{a}k extended some of the results of \cite{GW-79} and, among others,
proved the next theorem which we reformulate in terms convenient for our purpose
(see \cite[Theorem 6]{D}).

Let $(Z, \sig)$ be an infinite minimal cascade on a compact metric space $Z$. 
Let $\mathcal{G}$ be a Polish group with a dense path-wise connected subgroup.
Let $C(Z, \mathcal{G})$ denote the Polish space of continuous functions from
$Z$ into $\mathcal{G}$ (the {\em cocycles}) and let $B(Z, \mathcal{G})$ be the collection of 
{\em co-boundaries}; i.e. 
$$
 B(Z, \mathcal{G}) = 
 \{\psi \in C(Z, \mathcal{G}) : \psi(z) = \phi(\sig z) \phi(z)^{-1},
 \ {\text {for some}}\  \phi \in C(Z, \mathcal{G})\},
$$
and let $C'(Z, \mathcal{G}) = \ol{B(Z, \mathcal{G})}$.

\begin{thm}[Dirb\'{a}k]\label{Dir}  
There exists a dense $G_\del$ subset $\mathcal{R} \subset C'(Z, \mathcal{G})$
such that for every cocycle $\phi \in \mathcal{R}$
the corresponding Polish dynamical system
(skew product) defined on $Z \times \mathcal{G}$ by the formula 
$$
T_\phi(z, g) = (\sig z, \phi(z) g),
$$
is topologically transitive.
\end{thm}

\br

\begin{thm}\label{cas2}
There exists a metric minimal PI cascade (of order $3$) whose enveloping semigroup has 
$2^{\frak{c}}$ minimal left ideals.
\end{thm}

\begin{proof}
Again we are going to apply Proposition \ref{prop}, 
and this time also Theorem \ref{Dir}, 
with the following input:

$(Z,\sig)$ is a fixed metric minimal equicontinuous cascade (e.g. one can take $(Z,\sig)$ to be
an irrational rotation on the unit circle).
As above we let $Y_0 = S^2, Y = Y_0 \times S^1$, 
and consider the corresponding actions of $\mathcal{G}_0$ and $\mathcal{G}$.
We set $X_0 = Z \times Y_0, X = Z \times Y_0 \times S^1$ and define the map $\pi_X : X \to X_0$
as $\pi_X(z, y) = (z, \pi(y))$, where $\pi : (Y, \mathcal{G}) \to (Y_0, \mathcal{G}_0)$. 
Clearly property $(*)$ 
and extreme proximality 
hold for the action of $\mathcal{G}_0$ on $Y_0$.
Therefore, Proposition \ref{prop} 
applies, as well as Theorem \ref{Dir}.
We pick any $T = T_\phi$ in the intersection of
the residual sets provided by these theorems. 
We also fix a point $z_0 \in Z$
such that the orbit of the point $(z_0, e)$, where $e$ is the identity
element of $\mathcal{G}$, is dense in $Z \times \mathcal{G}$.

Given an infinite subset $A \subset Y_0$ with an infinite complement $B = Y_0 \setminus A$
we now consider the directed set $\{(E, F, n, \ep) : E \subset A, F\subset B, n \in \N, \ep >0 \}$,
where $(E , F, n, \ep) \prec (E' , F', n', \ep')$
iff $E \subset E', F \subset F'$, $n \le n'$ and $\ep' < \ep$.
For each $(E , F, n, \ep)$ we choose $\nu_{(E , F, n, \ep)} \in \Z$ such that 
(in $Z \times \mathcal{G}$)
$$
d(T_\phi^{\nu_{(E , F, n, \ep)}}(z_0,e), (z_0, g_{(E , F, n)})) < \ep.
$$ 
If we let $\phi_k(z) = \phi(\sig^{k-1}z) \phi(\sig^{k-2}z)\cdots \phi(\sig z)\phi(z)$,
then
$$
T_\phi^{\nu_{(E , F, n, \ep)}}(z_0,e) = (\sig^{\nu_{(E , F, n, \ep)}} z_0, \phi_{\nu_{(E , F, n, \ep)}}(z_0)).
$$
It now follows that along the directed set $\{(E , F, n, \ep)\}$ we have,
in $X = Z \times Y$,
\begin{gather*}
\lim T_\phi^{\nu_{(E , F, n, \ep)}}(z_0, y, \zeta) =(z_0, u_A(y, \zeta)),
\ {\text{and in}}\ E(Y, \mathcal{G}),\\
 \lim \phi_{\nu_{(E , F, n, \ep)}}(z_0) = u_A,
\end{gather*}
where
$$
u_A(y, \zeta) = \begin{cases} (y_0, -\zeta) & y \in A\\  (y_0, \zeta) & y \in B = A^c  \end{cases}
$$
Thus, denoting $q = \lim T_\phi^{\nu_{(E , F, n, \ep)}}$, the limit of the net $T_\phi^{\nu_{(E , F, n, \ep)}}$
in $E(X, T_\phi)$, we have
\begin{equation}\label{eq1}
q(z_0, y, \zeta) = (z_0, u_A(y, \zeta)) =  (z_0, y_0, \pm \zeta). 
\end{equation}

\br

Let now $v$ be any minimal idempotent in $E(X,T)$. 
Clearly $v$ restricts to $Z$ as the identity.
By Lemma \ref{skew},
for any $y \in Y$, $v(z_0, y) = (vz_0, v_{z_0}y) = (z_0, v_{z_0}y)$, where $v_{z_0}$
is an idempotent in $E(Y, \mathcal{G})$. We let $\pi_*v_{z_0}=\tilde{v}_{z_0}$
be the corresponding idempotent in $E(Y_0, \mathcal{G}_0)$.
Because $v$ is a minimal idempotent and because the extension 
$P : (X_0, \tilde{T}) \to (Z, \sig)$ is 
proximal, 
it follows that
$\tilde{v}_{z_0}$ is a minimal idempotent in $E(Y_0, \mathcal{G}_0)$;
i.e. its range is a singleton, one point set (in $Y_0 = S^2$). We can further require, as we may, that 
this single point be $y_0$.
We now have
\begin{equation}\label{eq2}
v_{z_0}(y, \zeta) = (y_0, \xi(y)\zeta), \qquad \ \forall\ (y, \zeta) \in Y = Y_0 \times S^1,
\end{equation}
for some function $\xi : Y_0 \to S^1$ and, as $u_{z_0}$ is an idempotent, we must have
$\xi(y_0) =1$.

Next consider the element $vq \in E(X, T_\phi)$. It belongs to some minimal left ideal, say
$I \subset   E(X, T_\phi)$ and (by (\ref{eq1}) and (\ref{eq2}))
\begin{gather*}
vq(z_0, y, \zeta) = v(z_0, y_0, \pm \zeta) = (z_0, y_0, \xi(y_0)( \pm \zeta)) = \\
 (z_0, y_0, \pm \zeta)
= q(z_0, y, \zeta) = (z_0, u_A(y, \zeta)).
\end{gather*}
Denote by $(vq)^{-1}$ the inverse of $vq$, in the maximal subgroup of $I$ which contains $vq$,
and let $\hat{u}_A = vq (vq)^{-1} = (vq)^{-1}vq$. Then $\hat{u}_A $ is a minimal idempotent in $I$ and
$$
\hat{u}_A(z_0, y, \zeta) = (z_0, u_A(y, \zeta)).
$$

Again it is easy to check that, restricted to the fiber of $X$ over $z_0$,
$\hat{u}_A \hat{u}_{A'}= \hat{u}_{A'}$ for any choice of sets 
$A, A' \subset Y_0$ as above.
This implies that for $A \not = A'$ the minimal idempotents $\hat{u}_A$ and  $\hat{u}_{A'}$
belong to different minimal left ideals in $E(X, T)$.
In fact, if they belong to the same minimal left ideal then, in $E(X,T)$ we have
$\hat{u}_A \hat{u}_{A'}= \hat{u}_{A}$, which, in turn, will imply
the same equality for the restriction on the fiber over $z_0$, leading to 
the equality $\hat{u}_A = \hat{u}_{A'}$ on this fiber, which is a contradiction. 
It follows that $E(X,T_\phi)$  has at least $2^{\frak{c}}$ minimal left ideals.
By cardinality arguments ${\rm{mi}}(X, T_\phi) \le 2^{\frak{c}}$ and our proof is complete. 
\end{proof}

\section{Appendix : A proof of Lemma \ref{idemp}}\label{Ap}

We first recall the following lemma of Furstenberg from \cite[Lemma 2]{Fu-76}.
Let $V$ be a finite dimensional real linear space. $P(V)$ will denote the corresponding projective space. 
If $v \in V$, $\bar{v}$ will denote the corresponding point of $P(V)$; if $W$ is a subspace of $V$, 
$\ol{W}$ will designate the corresponding linear subvariety of $P(V)$. 
Finite unions of linear subvarieties will be called {\em quasi-linear subvarieties}. 
As for all algebraic subvarieties, these satisfy the descending chain condition. This leads to:

\begin{lem}
 Let $\tau_n \in GL(V)$  and let $\bar{\tau}_n$ denote the corresponding projective transformations. 
 Assume $\det{\tau_n} =1$ and $\|\tau_n \| \to \infty$, where $\| \cdot \|$  
 is a suitable norm on the linear endomorphisms of $V$. 
 There exists a transformation $\pi$ of $P(V)$
whose range is a quasi-linear subvariety $\subsetneq P(V)$, and a sequence 
$\{n_k\}$ with $\bar{\tau}_{n_k}(x) \to \pi(x)$ for every $x \in P(V)$.
\end{lem}

In our case, where $\dim(V) =2$, the transformation $\pi$ is either of the form $\bar{g}$ with $g \in SL_2(\R)$,
or it is has the form $\pi = \pi_{x_0, x_1, x_2}$, where $\pi(x) = x_0$ for every $x \in \mathbb{P}^1 \setminus \{x_1\}$
and $\pi(x_1) =x_2$ (possibly with $x_0 = x_2$). Note that 
in this case the range of $\pi$ consists of at most two points.

\begin{proof}[Proof of Lemma \ref{idemp}]
(1)   
Consider first the action of $PSL_2(\R)$ on $Y = \mathbb{P}^1$.
If $\bar{g}y =y$ for $g \in PSL_2(\R)$ and every $y \in Y$ then $ge_i = \lambda_i e_i, \ i=1,2$,
with $\lambda_1 \lambda_2 =1$, so that 
$g=\begin{psmallmatrix}\lambda_1 & 0 \\ 0 & \lambda_2\end{psmallmatrix}$.
If $|\lambda_1| \not = |\lambda_2|$ then $g\begin{psmallmatrix}1 \\  1\end{psmallmatrix} =
\begin{psmallmatrix}\lambda_1 \\ \lambda_2\end{psmallmatrix}$ and
$\ol{\begin{psmallmatrix}1 \\  1\end{psmallmatrix}} \not = 
\ol{\begin{psmallmatrix}\lambda_1 \\ \lambda_2\end{psmallmatrix}}$,
which contradicts our assumption. Thus $|\lambda_1| = |\lambda_2|$, hence $g = \id$ in $PSL_2(\R)$.
We conclude that the map $g \mapsto \bar{g}$, from $PSL_2(\R)$ to $E(Y, PSL_2(\R))$ is a continuous injection.

Next assume that $\bar{g}_n y \to \bar{g} y$ for $g \in PSL_2(\R)$, a sequence $g_n \in PSL_2(\R)$,
and every $y \in Y$
\footnote{It follows from Furstenberg's lemma that the topological space $E(Y, PSL_2(\R))$
is Fr\'{e}chet; i.e. its topology is determined by sequences, rather than nets.
In other words, the dynamical system  $(Y, PSL_2(\R))$ is {\em tame}, see \cite{Gl-17}.}.
By Furstenberg's lemma (and its proof) we have 
pointwise convergence $g_n \to h$, with $h$ a semi-linear map $\R^2 \to \R^2$.
In particular $\bar{g}y = \bar{h}y$ for every $y \in Y$.
But, as observed above, if $h$ is not a linear map then $\bar{h}$ has range which consists of at most two points;
so it follows that $h$ is an invertible linear map. Now, the equality $\bar{g} = \bar{h}$ implies that $g = h$
as elements of $PSL_2(\R)$. Thus the map $g \mapsto \bar{g}$, from $PSL_2(\R)$ to $E(Y,PSL_2(\R))$ is a
homeomorphism.
Now, when dealing with $X$, the space of rays, instead of $Y = \mathbb{P}^1$,
we have the canonical surjective map $E(X, SL_2(\R)) \to E(Y, PSL_2(\R))$. 
Again the map $g \mapsto \tilde{g} \in E(X, SL_2(\R))$ is clearly a continuous injection and,
as we now in a situation where rays are mapped onto rays, it follows easily that the assumption
``$\tilde{g}_n x \to \tilde{g} x$ for $g \in SL_2(\R)$, a sequence $g_n \in SL_2(\R)$,
and every $x \in X$"  implies that $g_n \to g$ in $SL_2(\R)$. Thus $g \mapsto \tilde{g}$ is
indeed a homeomorphism from $SL_2(\R)$ onto its image in $E(X, SL_2(\R))$.

(2) This is easily checked.

(3) - (4)
It follows from the discussion so far that an element $p$ of $E(Y, PSL_2(\R))$
is either of the form $p = \bar{g}$ for some $g \in PSL_2(\R)$, or it has the form
$p = \bar{h}$ for some some semi-linear map $h$. If in addition $p^2 =p$ is an idempotent
then either $p = \id$ or it has the form $\pi = \pi_{y_0}$, where $\pi(y) = y_0$ for every
$y \in Y$.
Now observe that there are exactly two possible liftings of $p$ to minimal  idempotents
$u, u'$  in  $E(X, SL_2(\R))$
\footnote{This is because $X \cong S^1$ is $1$-dimensional; the situation is radically
different when one goes to $S^2$ and $\mathbb{P}^2$ (see the proof of Theorem \ref{cas2})}
as claimed. This proves part (4) and part (3) follows as well.
\end{proof}

\br

\br

\br


\begin{thebibliography}{10}


\bibitem{A-97}
Ethan  Akin, 6
{\em Recurrence in topological dynamics. Furstenberg families and Ellis actions}.
The University Series in Mathematics. 
Plenum Press, New York, 1997.

\bibitem{Au}
Joseph Auslander,
{\em Minimal Flows and their Extensions\/},
Mathematics Studies 153, Notas de Matem\'atica, 1988.

\bibitem{D}
Mat\'{u}\v{s} Dirb\'{a}k,  
{\em Minimal extensions of flows with amenable acting groups}, 
Israel J. Math. {\bf 207}, (2015),  581--615.

\bibitem{EGS}
Robert Ellis, Shmuel Glasner and Leonard Shapiro,
{\em Proximal-Isometric flows},
Advances in Math., {\bf 17},
(1975), 213--260.

\bibitem{Fu-76}
Harry Furstenberg,
{\em A note on Borel's density theorem},
Proc. Amer. Math. Soc.,
{\bf 55}, (1976), .209--212.

\bibitem{Gl-75}
Shmuel Glasner,
{\em A metric minimal flow whose enveloping semigroup contains finitely many minimal left ideals is PI},
Israel J. Math. {\bf 22},  (1975), 87--92.

\bibitem{Gl-PF}
Shmuel Glasner,
{\em Proximal flows},
Lecture Notes in Mathematics, Vol. {\bf 517}, Springer-Verlag, 
Berlin-New York, 1976.

\bibitem{Gl-03}
Eli Glasner, {\em Ergodic Theory via joinings\/},
Math. Surveys and
Monographs, AMS, {\bf 101}, 2003.

\bibitem{Gl-17}
Eli Glasner,
{\em The structure of tame minimal dynamical systems for general groups},
Invent. math.
DOI 10.1007/s00222-017-0747-z.

\bibitem{GW-79}
Eli Glasner and Benjamin Weiss,
{\em On the construction of minimal skew-products\/},
Israel J.\ of Math.\
{\bfseries 34}, (1979), 321--336.

 
 \bibitem{HJ-99}
Kamel N. Haddad and Aimee S. A. Johnson,
{\em  Recurrent sequences and IP sets},
 II Iberoamerican Conference on Topology and its Applications (Morelia, 1997). 
 Topology Appl. {\bf 98}, (1999),  203--210.
 
\bibitem{KR}
Aleksander Kechris and Christian Rosendal,
{\em Turbulence, amalgamation, and generic automorphisms of homogeneous structures},
Proc. Lond. Math. Soc. (3) {\bf 94}, (2007), 302--350.

\bibitem{McM}
Douglas McMahon,
{\em  Weak mixing and a note on the structure theorem for
minimal transformation groups},
Illinois J. of Math.,
{\bf 20}, (1976), 186--197.

\bibitem{M}
Duglas McMahon, 
{\em Relativized weak mixing of uncountable order}, 
Canad. J. Math. {\bf 32}, (1980), 559--566. 

%
%


\bibitem{S-17}
Petra Staynova,
{\em The Ellis Semigroup of a Generalised Morse System},
arXiv:1711.10484.



%
%

\bibitem{V}
William. A. Veech,
{\em Topological dynamics\/},
Bull. Amer. Math. Soc., {\bfseries 83},
(1977), 775--830.

\end{thebibliography}
\end{document}